\documentclass{amsart}
\usepackage{amsfonts,amsmath,amssymb}
\usepackage{hyperref}

\newtheorem*{theorem*}{Theorem}
\newtheorem*{corollary*}{Corollary}
\newtheorem{lemma}{Lemma}[subsection]
\newtheorem{corollary}[lemma]{Corollary}
\newtheorem{proposition}[lemma]{Proposition}

\newtheorem{theorem}[lemma]{Theorem}

\theoremstyle{definition}
\newtheorem{definition}[lemma]{Definition}

\theoremstyle{remark}
\newtheorem{notation}[lemma]{Notation}
\newtheorem{remark}[lemma]{Remark}
\newtheorem{example}[lemma]{Example}

\newcommand{\cc}{\mathbb{C}}

\newcommand{\Z}{{\mathbb Z}}

\newcommand{\R}{{\mathbb R}}
\newcommand{\C}{{\mathbb C}}

\newcommand{\proofend}{\hfill$\Box$\smallskip}

\newcommand{\pr}{{\operatorname{pr}}}

\newcommand{\Sc}{{\mathcal S}}
\newcommand{\G}{{\mathcal G}}

\newcommand{\Fre}{{Fr\'{e}chet \,}}
\newcommand{\tG}{{\widetilde{G}}}
\newcommand{\tH}{{\widetilde{H}}}

\newcommand{\Fou}{{\mathcal{F}}}

\newcommand{\cC}{{C_c^{\infty}}}
\newcommand{\cD}{{\mathcal{D}}}
\newcommand{\g}{{\mathfrak{g}}}
\newcommand{\Supp}{\mathrm{Supp}}

\makeatletter

\newcommand\classification[2][]{%
  \gdef\@classification{%
    \href{http://www.ams.org/msc/}%
{\textit{2000 Mathematics Subject Classification}}
\ignorespaces#2\unskip}}

\def\@classification{\CPM@warn{No MSC classification}\textit{2000
Mathematics Subject Classification} ????}
\def\received#1{%
  \gdef\@received{Received #1}}
\def\@received{\CPM@warn{No received date}Received 0000}

\makeatother

\begin{document}

\author{Avraham Aizenbud}
\address{Avraham Aizenbud and Dmitry Gourevitch, Faculty of Mathematics and Computer
Science, The Weizmann Institute of Science POB 26, Rehovot 76100,
Israel} \email{aizenr@yahoo.com}

\author{Dmitry Gourevitch} \email{guredim@yahoo.com}

\author{Eitan Sayag}
\address{Eitan Sayag, Institute of Mathematics, The Hebrew University of Jerusalem,
Jerusalem, 91904, Israel} \email{eitan.sayag@gmail.com}


%\classification{22E,22E45,20G05,20G25,46F99}
%\title[(GL(n+1,F),GL(n,F)) is a Gelfand pair for any local field F]
\title[$(\mathrm{GL}_{n+1}(F),\mathrm{GL}_{n}(F))$ is a Gelfand pair]
{$(\mathrm{GL}_{n+1}(F),\mathrm{GL}_{n}(F))$ is a Gelfand pair\\
for any local field $F$}
%\date{\today}
\begin{abstract}
Let $F$ be an arbitrary local field. Consider the standard embedding
$\mathrm{GL}_n(F) \hookrightarrow \mathrm{GL}_{n+1}(F)$ and the
two-sided action of $\mathrm{GL}_n(F) \times \mathrm{GL}_n(F)$ on
$\mathrm{GL}_{n+1}(F)$.

In this paper we show that any $\mathrm{GL}_n(F) \times
\mathrm{GL}_n(F)$-invariant distribution on $\mathrm{GL}_{n+1}(F)$
is invariant with respect to transposition.

We show that this implies that the pair $(GL_{n+1}(F),GL_{n}(F))$ is
a Gelfand pair. Namely, for any irreducible admissible
representation $(\pi,E)$ of $\mathrm{GL}_{n+1}(F)$,
$$\dim Hom_{\mathrm{GL}_n(F)}(E,\cc) \leq 1.$$

For the proof in the archimedean case we develop several tools to
study invariant distributions on smooth manifolds.
%The most important part of this paper is the archimedean case as
%well as tools developed to solve it.
\end{abstract}
\maketitle \tableofcontents
\section{Introduction}
Let $F$ be an arbitrary local field. Consider the standard
imbedding $\mathrm{GL}_n(F) \hookrightarrow \mathrm{GL}_{n+1}(F)$.
We consider the two-sided action of $\mathrm{GL}_n(F) \times
\mathrm{GL}_n(F)$ on $\mathrm{GL}_{n+1}(F)$ defined by
%$(g_1,g_2)h:=j(g_1)hj(g_2^{-1})$
$(g_1,g_2)h:=g_1 hg_2^{-1}$. In this paper we prove the following
theorem:
\begin{theorem*} [A]%{\bf A}.
Any $\mathrm{GL}_n(F) \times \mathrm{GL}_n(F)$ invariant
distribution on $\mathrm{GL}_{n+1}(F)$ is invariant with respect
to transposition.
\end{theorem*}
Theorem A has the following consequence in representation theory.

\begin{theorem*}[B]% {\bf B}.
Let $(\pi,E)$ be an irreducible admissible representation of
$\mathrm{GL}_{n+1}(F)$. Then
\begin{equation}\label{dim1}
\dim Hom_{\mathrm{GL}_n(F)}(E,\cc) \leq 1.
\end{equation}
\end{theorem*}

Since any character of $\mathrm{GL}_{n}(F)$ can be extended to
$\mathrm{GL}_{n+1}(F)$, we obtain

\begin{corollary*}% {\bf B}.
Let $(\pi,E)$ be an irreducible admissible representation of
$\mathrm{GL}_{n+1}(F)$ and let $\chi$ be a character of
$\mathrm{GL}_{n}(F)$.  Then
$$\dim Hom_{\mathrm{GL}_n(F)}(\pi,\chi) \leq 1.$$
\end{corollary*}

In the non-archimedean case we use the standard notion of
admissible representation (see \cite{BZ}). In the archimedean case
we consider admissible smooth \Fre representations (see section
\ref{repth}).
%will specify what we mean by admissible representation in
%section \ref{repth}.

Theorem B has some application to the theory of automorphic forms,
more specifically to the factorizability of certain periods of
automorphic forms on $GL_{n}$ (see \cite{Flicker} and
\cite{Flicker-Nikolov}).

We deduce Theorem B from Theorem A using an argument due to
Gelfand and Kazhdan adapted to the archimedean case. In our
approach we use two deep results: the globalization theorem of
Casselman-Wallach (see \cite{WallachB2}), and the regularity
theorem of Harish-Chandra (\cite{WallachB1}, chapter 8).

Clearly, Theorem B implies in particular that (\ref{dim1}) holds
for unitary irreducible representations of $\mathrm{GL}_{n+1}(F)$.
That is, the pair $(\mathrm{GL}_{n+1}(F),\mathrm{GL}_{n}(F))$ is a
generalized Gelfand pair in the sense of \cite{vD} and
\cite{Bos-vD}.

The notion of Gelfand pair was studied extensively in the
literature both in the setting of real groups and $p$-adic groups
(e.g. \cite{Gelfand-Kazhdan}, \cite{vD}, \cite{vD-P},
\cite{Bos-vD}, \cite{Gross}, \cite{Prasad} and \cite{JR} to
mention a few). In \cite{vD}, the notion of generalized Gelfand
pair is defined by requiring a condition of the form (\ref{dim1})
for irreducible unitary representations. The definition suggested
in \cite{Gross} refers to the non-archimedean case and to a
property satisfied by all irreducible admissible representations. % of $G(F)$.
In both cases, the verification of the said condition is achieved
by means of a theorem on invariant distributions. However, the
required statement on invariant distributions needed to verify
condition (\ref{dim1}) for unitary representation concerns only
positive definite distributions. We elaborate on these issues in
section \ref{repth}.
%One of the purposes of this paper is to put the two cases on equal
%footing, namely to show that
%theory of generalized Gelfand pairs was developed in \cite....
\subsection{Related results}
$ $

Several existing papers study related problems.

The case of non-archimedean fields of zero characteristic is
covered in \cite{AGRS} (see also \cite{AG2}) where it is proven
that the pair $(GL_{n+1}(F),GL_{n}(F))$ is a strong Gelfand pair
i.e. $\dim_{H}(\pi,\sigma) \leq 1$ for any irreducible admissible
representation $\pi$ of $G$ and \emph{any} irreducible admissible
representation $\sigma$ of $H$. Here $H=\mathrm{GL}_{n}(F)$ and
$G=\mathrm{GL}_{n+1}(F)$.

In \cite{JR}, it is proved that
$(\mathrm{GL}_{n+1}(F),\mathrm{GL}_{n}(F) \times
\mathrm{GL}_{1}(F))$ is a Gelfand pair, where $F$ is a local
non-archimedean field of zero characteristic.

In \cite{vD-P} it is proved that for $n \geq 2$ the pair
$(SL_{n+1}(\R),GL_{n}(\R))$ is a generalized Gelfand pair %, where $F$ is the field of real numbers
and a similar result is obtained in \cite{Bos-vD} for the $p$-adic
case, for $n \geq 3$. We emphasize that these results are proved
in the realm of unitary representations. Another difference
between these works and the present paper is that the embedding
$GL_{n}(F) \subset GL_{n+1}(F)$ studied here does not factor
through the embedding $GL_{n}(F) \hookrightarrow SL_{n+1}(F)$ of
\cite{vD-P}. In particular, $(GL_{2}(\R),GL_{1}(\R))$ is a
generalized Gelfand pair, and the pair $(SL_{2}(\R),GL_{1}(\R))$
is not a generalized Gelfand pair (\cite{Molchanov},\cite{vD}).

%Neither of these results implies the other but they are all very
%closely related.

%\subsection{main idea and contribution}
%ormal reduction to distributions in the context of \Fre rep.;
%study of distributions: Schwarz and indeed any.

%\subsection{relation to other works}
%maybe mention flicker.

\subsection{Content of the Paper}
$ $

We now briefly sketch the structure and content of the paper.

In section \ref{repth} we prove that Theorem A implies Theorem B.
For this we clarify the relation between the theory of Gelfand
pairs and the theory of invariant distributions both in the
setting of \cite{vD} and in the setting of \cite{Gross}.

In section \ref{p-adic} we present the proof of theorem A in the
non-archimedean case. This section gives a good introduction to
the rest of the paper since it contains many of the ideas but is
technically simpler.

In section \ref{Prel} we provide several tools to study invariant
distributions on smooth manifolds.
% Some of these results are
%noveland, we believe, are of independent interest.
We believe that these results are of independent interest. In
particular we introduce
%Frobenius reciprocity (Theorem \ref{Frob})
%an archimedean version of Bernstein's localization principle
%(Theorem \ref{Localization}) and
an adaption of a trick due to Bernstein which is very useful in
the study of invariant distributions on vector spaces (proposition
\ref{SimpleOrbitCheck}).
%Bruhat change
%In addition we prove that under certain
%conditions an equivariant distribution is {\it a-priori} a
%Schwartz distribution (proposition \ref{InvAreSchwartz}).
These
results partly relay on \cite{AG}.

In section \ref{proofs} we
prove Theorem A in the archimedean case. This is the main result of the paper. % \ref{goal}.
The scheme of the proof is similar to the non-archimedean case.
However, it is complicated by the fact that distributions on real
manifolds do not behave as nicely as distributions on
$\ell$-spaces (see \cite{BZ}).

We now explain briefly the main difference between the study of
distributions on $\ell$-spaces and distributions on real
manifolds.

The space of distributions on an $\ell$-space $X$ supported on a
closed subset $Z \subset X$ coincides with the space of
distributions on $Z$. In the presence of group action on $X$, one
can frequently use this property to reduce the study of
distributions on $X$ to distributions on orbits, that is on
homogenous spaces. Although this property fails for distributions
on real manifolds, one can still reduce problems to orbits. In the
case of finitely many orbits this is studied in \cite{Bruhat},
\cite{CHM}, \cite{AG}. %In the present work we need to extend these
%techniques to the case where the number of orbits is infinite. We
%achieve this through a localization principle which is an
%archimedean analogue of Bernstein's localization principle
%described in \cite{Ber}, section 1.4.
%
%Problem - LocPrin mentioned !!! Double-plus ungood ???

We mention that unlike the $p$-adic case, after the reduction to
the orbits one needs to analyze generalized sections of symmetric
powers of the normal bundles to the orbits, and not just
distributions on those orbits. Here we employ a trick, proposition
\ref{Trick}, which allows us to recover this information from a
study of invariant distributions on a larger space.
%Since such results are proved in the $p$-adic case

In section \ref{SecFrob} we provide the proof for the Frobenius
reciprocity.   The proof follows the proof in \cite{Bar} (section
3).

In section \ref{AppFilt} we prove the rest of the statements of
section \ref{Prel}.

\subsection*{Acknowledgements}
This work was conceived while the three authors were visiting at
the Hausdorff Institute of Mathematics (HIM) at Bonn while
participating in the program \emph{Representation theory, complex
analysis and integral geometry} joint with Max Planck Institute
fur Mathematik.

We wish to thank the director of the HIM, \textbf{Prof. Matthias
Kreck}, for the inspiring environment and perfect working
conditions at the HIM.

We wish to thank \textbf{Prof. Gerrit van Dijk} for useful e-mail
correspondence. We thank \textbf{Prof. Bernhard Kroetz} for useful
advice and \textbf{Dr. Oksana Yakimova} for fruitful discussions.

Finally, we thank our teacher \textbf{Prof. Joseph Bernstein} for
our mathematical education.

During the preparation of this work, Eitan Sayag was partially
supported by ISF grant number 147/05.

\section{Generalized Gelfand pairs and invariant
distributions}\label{repth}

In this section we show that Theorem A implies Theorem B. When $F$
is non-archimedean this is a well known argument of Gelfand and
Kazhdan (see \cite{Gelfand-Kazhdan, Prasad}). When $F$ is
archimedean and the representations in question are unitary such a
reduction
%is well known and
is due to \cite{Thomas}. We wish to consider representations which
are not necessarily unitary and present here an argument which is
valid in the generality of admissible smooth \Fre representations.
Our treatment is close in spirit to \cite{Shalika} (where
multiplicity one result of Whittaker model is obtained for unitary
representation) but at a crucial point we need to use the
globalization theorem of Casselman-Wallach.

\subsection{Smooth \Fre representations}
$ $

The theory of representations in the context of \Fre spaces is
developed in \cite{Cas2} and \cite{WallachB2}. We present here a
well-known slightly modified version of that theory.

\begin{definition}
Let $V$ be a complete locally convex topological vector space. A
representation $(\pi,V,G)$ is a continuous map $G \times V \to V$.
A representation is called \textbf{\Fre} if there exists a
countable family of semi-norms $\rho_{i}$ on $V$ defining the
topology of $V$ and such that the action of $G$ is continuous with
respect to each $\rho_{i}$. We will say that $V$ is \textbf{smooth
\Fre} representation if, for any $X \in \g$ the differentiation
map $ v \mapsto \pi(X)v$ is a continuous linear map from $V$ to
$V$.

\end{definition}

An important class of examples of smooth \Fre representations is
obtained from continuous Hilbert representations $(\pi,H)$ by
considering the subspace of smooth vectors $H^{\infty}$ as a \Fre
space (see \cite{WallachB1} section 1.6 and \cite{WallachB2} 11.5).

We will consider mostly smooth \Fre representations.

\begin{remark}
In the language of \cite{WallachB2} and \cite{Cas} the
representations above are called smooth \Fre representations of
moderate growth.
\end{remark}
%????Rami

Recall that a smooth \Fre representation is called
\textit{admissible} if it is finitely generated and its underlying
$(\g,K)$-module is admissible. In what follows \textit{admissible
representation} will always refer to admissible smooth \Fre
representation.

For a smooth admissible \Fre representation $(\pi,E)$ we denote by
$(\widetilde{\pi},\widetilde{E})$ the smooth contragredient of
$(\pi,E)$.
%
%define the smooth
%dual $(\widetilde{\pi},\widetilde{E})$ by
%$$\widetilde{E}={E^{*}}^{\infty},$$
%where $E^{*}$ is the space of all continuous functionals on $E$.
%Here ${E^{*}}^{\infty}$ is a \Fre space consisting of the smooth
%vectors of $E^{*}$.
%????Rami

We will require the following corollary of the globalization
theorem of Casselman and Wallach (see \cite{WallachB2} , chapter
11).
\begin{theorem}\label{cortoCW}
Let $E$ be an admissible \Fre representation, then there exists a
continuous Hilbert space representation $(\pi,H)$ such that
$E=H^{\infty}$.
\end{theorem}
This theorem follows easily from the embedding theorem of
Casselman combined with Casselman-Wallach globalization theorem.

\Fre representations of $G$ can be lifted to representations of
$\Sc(G)$, the Schwartz space of $G$. This is a space consisting of
functions on $G$ which, together with all their derivatives, are
rapidly decreasing (see \cite{Cas}. For an equivalent definition
see section \ref{Notations}).

For a \Fre representation $(\pi,E)$ of $G$, the algebra $\Sc(G)$
acts on $E$ through
\begin{equation}\label{operator}
\pi(\phi)=\int_{G} \phi(g) \pi(g)dg
\end{equation}
(see \cite{WallachB1}, section 8.1.1).\\
The following lemma is straitforward:

\begin{lemma}\label{generalShalika}
Let $(\pi,E)$ be an admissible \Fre representation of $G$ and let
$\lambda \in E^*$. Then $\phi \to \pi(\phi)\lambda$ is a
continuous map $\Sc(G) \to \widetilde{E}$.
\end{lemma}

The following proposition follows from Schur's lemma for $(\g,K)$
modules (see \cite{WallachB1} page 80) in light of
Casselman-Wallach theorem.

\begin{proposition}\label{SchurCW}
Let $G$ be a real reductive group. Let $W$ be a \Fre representation
of $G$ and let $E$ be an irreducible admissible representation of
$G$. Let $T_{1},T_{2}:W \hookrightarrow E$ be two embeddings of $W$
into $E$. Then $T_{1}$ and $T_{2}$ are proportional.
\end{proposition}

We need to recall the basic properties of characters of
representations.

\begin{proposition}\label{character}
Assume that $(\pi,E)$ is admissible \Fre representation. Then
$\pi(\phi)$ is of trace class, and the assignment $\phi \to
trace(\pi(\phi))$ defines a continuous functional on $\Sc(G)$ i.e.
a Schwartz distribution. Moreover, the distribution
$\chi_{\pi}(\phi)=trace(\pi(\phi))$ is given by a locally
integrable function on $G$.
\end{proposition}

The result is well known for continuous Hilbert representations
(see \cite{WallachB1} chapter 8). The case of admissible \Fre
representation follows from the case of Hilbert space
representation and theorem \ref{cortoCW}.

Another useful property of the character (see loc. cit.) is the
following proposition:
\begin{proposition}
If two irreducible admissible representations have the same
character then they are isomorphic.
\end{proposition}

\begin{proposition} Let $(\pi,E)$ be an admissible
representation. Then $\widetilde{\widetilde{E}} \cong E$.
\end{proposition}
For proof see pages 937-938 in \cite{GP}.

\subsection{Three notions of Gelfand pair}
$ $

Let $G$ be a real reductive group and $H \subset G$ be a subgroup.
Let $(\pi,E)$ be an admissible \Fre representation of $G$ as in
the previous section. We are interested in representations
$(\pi,E)$ which admit a continuous $H$-invariant linear
functional. Such representations of $G$ are called
$H$-distinguished.

Put differently, let $Hom_{H}(E,\cc)$ be the space of continuous
functionals $\lambda :E \to \cc$ satisfying
$$\forall e \in E, \forall h \in H: \lambda(he)=\lambda(e)$$
The representation $(\pi,E)$ is called {\bf $H$-distinguished} if
$Hom_{H}(E,\cc)$ is non-zero. We now introduce three notions of
Gelfand pair and study their inter-relations.

\begin{definition}\label{three notions}
Let $H \subset G$ be a pair of reductive groups.
\begin{itemize}
\item We say that $(G,H)$ satisfy {\bf GP1} if for any irreducible
admissible representation $(\pi,E)$ of $G$ we have
$$\dim Hom_{H}(E,\cc) \leq 1$$
\item We say that $(G,H)$ satisfy {\bf GP2} if for any irreducible
admissible representation $(\pi,E)$ of $G$ we have
%\begin{equation}\label{dim}
$$\dim Hom_{H}(E,\cc) \cdot \dim Hom_{H}(\widetilde{E},\cc)\leq
1$$

\item We say that $(G,H)$ satisfy {\bf GP3} if for any irreducible
{\bf unitary} representation $(\pi,W)$ of $G$ on a Hilbert space
$W$ we have
%\begin{equation}\label{dim}
$$\dim Hom_{H}(W^{\infty},\cc) \leq 1$$
%\end{equation}
\end{itemize}

\end{definition}
Property GP1 was established by Gelfand and Kazhdan in certain
$p$-adic cases (see \cite{Gelfand-Kazhdan}). Property GP2 was
introduced by \cite{Gross} in the $p$-adic setting. Property GP3
was studied extensively by various authors under the name {\bf
generalized Gelfand pair} both in the real and $p$-adic settings
(see e.g. \cite{vD-P}, \cite{Bos-vD}).

We have the following straitforward proposition:

\begin{proposition}
$GP1 \Rightarrow GP2 \Rightarrow GP3.$
\end{proposition}

\subsection{Gelfand pairs and invariant distributions}
$ $

The theory of generalized Gelfand pairs as developed in
\cite{vD-P} and \cite{Thomas} provides the following criterion to
verify GP3.

\begin{theorem}
Let $\tau$ be an involutive anti-automorphism of $G$ such that
$\tau(H)=H$. Suppose $\tau(T)=T$ for all bi $H$-invariant positive
definite distributions $T$ on $G$. Then $(G,H)$ satisfies GP3.
\end{theorem}
\noindent This is a slight reformulation of Criterion 1.2 of
\cite{vD}, page 583.

We now consider an analogous criterion which allows the
verification of GP2. This is inspired by the famous
Gelfand-Kazhdan method in the $p$-adic case.

\begin{theorem}\label{GKThomasCW}
Let $\tau$ be an involutive anti-automorphism of $G$ and assume
that $\tau(H)=H$. Suppose $\tau(T)=T$ for all bi $H$-invariant
distributions \footnote{In fact it is enough to check this only
for Schwartz distributions.} on $G$. Then $(G,H)$ satisfies GP2.
\end{theorem}

\begin{proof}
Let $(\pi,E)$ be an irreducible admissible \Fre representation. If
$E$ or $\widetilde{E}$ are not distinguished by $H$ we are done.
Thus we can assume that there exists a non-zero $\lambda:E \to
\cc$ which is $H$-invariant. Now let $\ell_{1},\ell_{2}$ be two
non-zero $H$-invariant functionals on $\widetilde{E}$. We wish to
show that they are proportional. For this we define two
distributions $D_{1},D_{2}$ as follows
$$D_{i}(\phi)=\ell_{i}(\pi(\phi)\lambda)$$ for $i=1,2$.
Here $\phi \in \Sc(G)$. Note that $D_{i}$ are also Schwartz
distributions. Both distributions are bi-$H$-invariant and hence,
by the assumption, both distributions are $\tau$ invariant. Now
consider the bilinear forms on $\Sc(G)$ defined by
$$B_{i}(\phi_{1},\phi_{2})=D_{i}(\phi_{1}*\phi_{2}).$$
Since $E$ is irreducible, the right kernel of $B_{1}$ is equal to
the right kernel of $B_{2}$. We now use the fact that $D_{i}$ are
$\tau$ invariant. Denote by $J_{i}$ the left kernels of $B_{i}$.
Then $J_{1}=J_{2}$ which we denote by $J$. Consider the \Fre
representation $W=\Sc(G)/J$ and define the maps $T_{i}: \Sc(G) \to
\widetilde{\widetilde{E}} \cong E$ by
$T_{i}(\phi)=\pi(\phi)\ell_{i}$. These are well defined by Lemma
\ref{generalShalika} and we use the same letters to denote the
induced maps $T_{i}: W \to E$. By proposition \ref{SchurCW},
$T_{1}$ and $T_{2}$ are proportional and hence $\ell_{1}$ and
$\ell_{2}$ are proportional and the proof is complete.
\end{proof}

\subsection{Archimedean analogue of Gelfand-Kazhdan's theorem}
$ $

To finish the proof that Theorem A implies Theorem B we will show
that in certain cases, the property $GP1$ is equivalent to $GP2$.

\begin{proposition}
Let $H < \mathrm{GL}_{n}(F)$ be a transposition invariant
subgroup. Then $GP1$ is equivalent to $GP2$ for the pair
$(\mathrm{GL}_{n}(F),H)$.
\end{proposition}
For the proof we need the following notation. For a representation
$(\pi,E)$ of $GL_{n}(F)$ we let $(\widehat{\pi},E)$ be the
representation of $GL_{n}(F)$ defined by $\widehat{\pi}=\pi \circ
\theta$, where $\theta$ is the (Cartan) involution $\theta(g)=
{g^{-1}}^t$. Since
$$Hom_{H}(\pi,\cc)=Hom_{H}({\widehat{\pi}},\cc)$$ the
following analogue of Gelfand-Kazhdan theorem is enough.
\begin{theorem}\label{GKreal}
Let $(\pi,E)$ be an irreducible admissible representation of
$GL_{n}(F)$. Then $\widehat{\pi}$ is isomorphic to
$\widetilde{\pi}$.
\end{theorem}

\begin{remark}
This theorem is due to Gelfand and Kazhdan in the $p$-adic case
(they show that any distribution which is invariant to conjugation
is transpose invariant, in particular this is valid for the
character of an irreducible representation) and due to Shalika for
unitary representations which are generic (\cite{Shalika}). We
give a proof in complete generality based on Harish-Chandra
regularity theorem (see chapter 8 of \cite{WallachB1}).
%Another argument - based on global method
%will be sketched bellow.
\end{remark}

\begin{proof}[Proof of theorem \ref{GKreal}]
Consider the characters $\chi_{\widetilde{\pi}}$ and
$\chi_{\widehat{\pi}}$. These are locally integrable functions on
$G$ that are invariant with respect to conjugation. Clearly,
$$\chi_{\widehat{\pi}}(g)=\chi_{\pi}({g^{-1}}^t)$$ and
$$\chi_{\widetilde{\pi}}(g)=\chi_{\pi}(g^{-1}).$$ But
for $g \in \mathrm{GL}_{n}(F)$, the elements $g^{-1}$ and
${g^{-1}}^t$ are conjugate. Thus, the characters of
$\widehat{\pi}$ and $\widetilde{\pi}$ are identical. Since both
are irreducible, Theorem 8.1.5 in \cite{WallachB1}, implies that
$\widehat{\pi}$ is isomorphic to $\widetilde{\pi}$.
\end{proof}

\begin{corollary}
Theorem A implies Theorem B.
\end{corollary}

\begin{remark}
The above argument proves also that Theorem B follows from a
weaker version of Theorem A, where only Schwartz distributions are
considered (these are continuous functionals on the space $\Sc(G)$
of Schwartz functions).
\end{remark}

\begin{remark}
The non-archimedean analogue of theorem \ref{GKThomasCW} is a
special case of Lemma 4.2 of \cite{Prasad}. The rest of the
argument in the non-archimedean case is identical to the above.
\end{remark}

\section{Non-archimedean case} \label{p-adic}

In this section $F$ is a non-archimedean local field of arbitrary
characteristic. We will use the standard terminology of $l$-spaces
introduced in \cite{BZ}, section 1. We denote by $\Sc(X)$ the
space of Schwartz functions on an $l$-space $X$, and by $\Sc^*(X)$
the space of distributions on $X$ equipped with the weak topology.

We fix a nontrivial additive character $\psi$ of $F$.
\subsection{Preliminaries}

\begin{definition} \label{DefCone}
Let $V$ be a finite dimensional vector space over $F$. A subset $C
\subset V$ is called a \textbf{cone} if it is homothety invariant.
\end{definition}

\begin{definition} \label{DefHomType}
Let $V$ be a finite dimensional vector space over $F$. Note that
$F^{\times}$ acts on $V$ by homothety. This gives rise to an
action $\rho$ of $F^{\times}$ on $\Sc^*(V)$.  Let $\alpha$ be a
character of $F^{\times}$.

We call a distribution $\xi \in \Sc^*(V)$ \textbf{homogeneous of
type $\alpha$} if for any $t \in F^{\times}$, we have
$\rho(t)(\xi)=\alpha^{-1}(t)\xi$. That is, for any function $f \in
\Sc(V)$, $\xi(\rho(t^{-1})(f)) = \alpha(t) \xi(f)$, where
$\rho(t^{-1})(f)(v)= f(tv)$.

Let $L subset F$ be a subfield. We will call a distribution $\xi
\in \Sc^*(V)$ \textbf{$L$-homogeneous of type $\alpha$} if for any
$t \in L^{\times}$, we have $\rho(t)(\xi)=\alpha^{-1}(t)\xi$.
\end{definition}

\begin{example}
A Haar measure on $V$ is homogeneous of type $| \cdot |^{\dim V}$.
The Dirac's $\delta$-distribution is homogeneous of type $1.$
\end{example}
The following proposition is straightforward.
\begin{proposition} \label{padic-Strat}
Let a $l$-group $G$ act on an $l$-space $X$. Let $X =
\bigcup_{i=0}^l X_i$ be a $G$-invariant stratification of $X$. Let
$\chi$ be a character of $G$. Suppose that for any $i = 1 \ldots l
$, $\Sc^*(X_i)^{G,\chi}=0$. Then $\Sc^*(X)^{G,\chi}=0$.
\end{proposition}

\begin{proposition} \label{padic-Product}
Let $H_i \subset G_i$ be $l$-groups acting on $l$-spaces $X_i$ for
$i=1 \ldots n$. Suppose that $\Sc^*(X_i)^{H_i}=\Sc^*(X_i)^{G_i}$
for all $i$. Then $\Sc^*(\prod X_i)^{\prod H_i}=\Sc^*(\prod
X_i)^{\prod G_i}$.
\end{proposition}
%For proof see subsection \ref{ProofProd}.
%\begin{proof}
\begin{proof}
It is enough to prove the proposition for the case $n=2$. Let $\xi
\in \Sc^*(X_1 \times X_1)^{H_1 \times H_2}$. Fix $f_1 \in
\Sc(X_1)$ and $f_2 \in \Sc(X_1)$. It is enough to prove that for
any $g_1 \in G_1$ and $g_2 \in G_2$ , we have $\xi(g_1(f_1)
\otimes g_2(f_2))= \xi(f_1 \otimes f_2)$. Let $\xi_1 \in
\Sc^*(X_1)$ be the distribution defined by $\xi_1(f) :=\xi(f
\otimes f_2).$ It is $H_1$-invariant. Hence also $G_1$-invariant.
Thus $\xi(f_1 \otimes f_2)= \xi(g_1(f_1) \otimes f_2)$. By the
same reasons $\xi(g_1(f_1) \otimes f_2)= \xi(g_1(f_1) \otimes
g_2(f_2))$.
%\end{proof}
\end{proof}

We will use the following important theorem proven in \cite{Ber},
section 1.5.
\begin{theorem}[Frobenius reciprocity] \label{padic-Frob}
Let a unimodular $l$-group $G$ act transitively on an $l$-space
$Z$. Let $\varphi:X \to Z$ be a $G$-equivariant continuous map.
Let $z\in Z$. Suppose that its stabilizer $\mathrm{Stab}_G(z)$ is
unimodular. Let $X_z$ be the fiber of $z$. Let $\chi$ be a
character of $G$. Then $\Sc^*(X)^{G,\chi}$ is canonically
isomorphic to $\Sc^*(X_z)^{\mathrm{Stab}_G(z),\chi}$.
\end{theorem}

The next proposition formalizes an idea from \cite{BeLect}. The
key tool used in its proof is Fourier Transform.

\begin{proposition} \label{padic-OrbitCheck}
Let $G$ be an $l$-group. Let $V$ be a finite dimensional
representation of $G$ over $F$. Suppose that the action of $G$
preserves some non-degenerate bilinear form on $V$. Let $V =
\bigcup \limits _{i=1}^n C_i$ be a stratification of $V$ by
$G$-invariant cones.

Let $\mathfrak{X}$ be a set of characters of $F^{\times}$ such
that the set $\mathfrak{X} \cdot \mathfrak{X}$ does not contain
the character $|\cdot|^{\dim V}$. Let $\chi$ be a character of
$G$. Suppose that for any $i$, the space $\Sc^*(C_i)^{G,\chi}$
consists of homogeneous distributions of type $\alpha$ for some
$\alpha \in \mathfrak{X}$. Then $\Sc^*(V)^{G,\chi}=0$.
\end{proposition}
In section \ref{ProofOCheck} we prove an archimedean analog of
this proposition, and the same proof is applicable in this case.

\subsection{Proof of Theorem A for non-archimedean $F$}\label{padic-proofs}
$ $

We need some further notations.
\begin{notation} \label{NotObjects}
Denote $H:=H_n:=\mathrm{GL}_n:=\mathrm{GL}_n(F)$. Denote
$$G:=G_n:=\{(h_1,h_2)\in \mathrm{GL}_n \times \mathrm{GL}_n|
\, det(h_1)=det(h_2)\}.$$ We consider $H$ to be diagonally
embedded to $G$.

Consider the action of the 2-element group $S_2$ on $G$ given by
the involution $(h_1,h_2) \mapsto ({h_2^{-1}}^t,{h_1^{-1}}^t)$. It
defines a semidirect product $\widetilde{G} := \widetilde{G}_n :=
G \rtimes S_2$. Denote also $\tH:= \tH_n:= H_n \rtimes S_2$.

Let $V=F^n$ and $X:=X_n := \mathrm{gl}_{n}(F) \times V \times
V^*$.

The group $\tG$ acts on $X$ by
$$(h_1,h_2)(A,v,\phi):= (h_1Ah_2^{-1},h_1v, {h_2^{-1}}^t \phi) \text{ and }$$
$$\sigma(A,v,\phi):=(A^t,\phi^t,v^t)$$
where $(h_1,h_2) \in G$ and $\sigma$ is the generator of $S_2$.
Note that $\tG$ acts separately on $\mathrm{gl}_n$ and on $V
\times V^*$. Define a character $\chi$ of $\widetilde{G}$ by
$\chi(g,s):= sign(s)$.
\end{notation}

We will show that the following theorem implies Theorem A.
\begin{theorem} \label{padic-main}
$\Sc^*(X)^{\tG,\chi}=0$.
\end{theorem}

\subsubsection{Proof that theorem \ref{padic-main} implies theorem
A} \label{SecReduction}

$ $

We will divide this reduction to several propositions.\\
Consider the action of $\tG_n$ on $\mathrm{GL}_{n+1}$ and on
$\mathrm{gl}_{n+1}$, where $G_n$ acts by the two-sided action and
the generator of $S_2$ acts by transposition.

\begin{proposition} \label{padic-Red1}
If  $\Sc^*(\mathrm{GL}_{n+1})^{\tG_n,\chi}=0$ then theorem A
holds.
\end{proposition}
The proof is straightforward.
\begin{proposition} \label{padic-Red2}
If  $\Sc^*(\mathrm{gl}_{n+1})^{\tG_n,\chi}=0$ then
$\Sc^*(\mathrm{GL}_{n+1})^{\tG_n,\chi}=0$.
\end{proposition}
\emph{Proof.}\footnote{This proposition is an adaption of a
statement in \cite{Ber}, section 2.2.} Let $\xi \in
\Sc^*(\mathrm{GL}_{n+1})^{\tG_n,\chi}$. We have to prove $\xi=0$.
Assume the contrary. Take $p \in \mathrm{Supp}(\xi)$. Let
$t=\mathrm{det}(p)$. Let $f\in \Sc(F)$ be such that $f$ vanishes
in a neighborhood of zero and $f(t) \neq 0$. Consider the
determinant map $\mathrm{det}:\mathrm{GL}_{n+1} \to F$. Consider
$\xi':=(f \circ \mathrm{det})\cdot \xi$. It is easy to check that
$\xi' \in \Sc^*(\mathrm{GL}_{n+1})^{\tG_n,\chi}$ and $p \in
\mathrm{Supp}(\xi')$. However, we can extend $\xi'$ by zero to
$\xi'' \in \Sc^*(\mathrm{gl}_{n+1})^{\tG_n,\chi}$, which is zero
by the assumption. Hence $\xi'$ is also zero. Contradiction.
\proofend
\begin{proposition} \label{padic-Red4}
If $\Sc^*(X_{n})^{\tG_n,\chi}=0$ then
$\Sc^*(\mathrm{gl}_{n+1})^{\tG_n,\chi}=0$.
\end{proposition}
\begin{proof}
Note that $\mathrm{gl}_{n+1}$ is isomorphic as a
$\tG_n$-equivariant $l$-space to $X_n \times F$ where the action
on $F$ is trivial. This isomorphism is given by
$$ \left(
  \begin{array}{cc}
    A_{n \times n} & v_{n\times 1} \\
    \phi_{1\times n} & t\\
  \end{array}
\right) \mapsto ((A , v ,\phi),t) .$$ The proposition now follows
from proposition \ref{padic-Product}.
\end{proof}
This finishes the proof that theorem \ref{padic-main} implies
Theorem A.
\subsubsection{Proof of theorem \ref{padic-main}}

$ $

We will now stratify $X (=gl_n \times V \times V^*)$ and deal with
each strata separately.
\begin{notation}
Denote $W:=W_n:=V_n \oplus V_n^*.$ Denote by $Q^i:=Q^i_n \subset
\mathrm{gl}_n$ the set of all matrices of rank $i$. Denote
$Z^i:=Z^i_n:=Q^i_n \times W_n$.
\end{notation}
Note that $X = \bigcup Z^i$. Hence by proposition
\ref{padic-Strat}, it is enough to prove the following
proposition.

\begin{proposition} \label{padic-Enough}
$\Sc^*(Z^i)^{\tG,\chi}=0$ for any $i=0,1,...,n$.
\end{proposition}

We will use the following key lemma.
\begin{lemma} [Non-archimedean Key Lemma] \label{padic-key}
$\Sc^*(W)^{\tH,\chi}=0.$
\end{lemma}
For proof see section \ref{padic-ProofKey} below.
\begin{corollary} \label{padic-in}
Proposition \ref{padic-Enough} holds for $i=n$.
\end{corollary}
\begin{proof}
Clearly, one can extend the actions of $\widetilde{G}$ on $Q^n$
and on $Z^n$ to actions of $\widetilde{GL_n \times GL_n}:= (GL_n
\times GL_n) \rtimes S_2$ in the obvious way.

Step 1. $\Sc^*(Z^n)^{\widetilde{GL_n \times GL_n},\chi}=0$.\\
Consider the projection on the first coordinate from $Z^n$ to the
transitive $\widetilde{GL_n \times GL_n}$-space  $Q^n=GL_n$.
Choose the point $Id \in Q^n$. Its stabilizer is $\tH$ and its
fiber is $W$. Hence by Frobenius reciprocity (theorem
\ref{padic-Frob}), $\Sc^*(Z^n)^{\widetilde{GL_n \times GL_n},\chi}
\cong  \Sc^*(W)^{\tH,\chi}$ which is zero by the key lemma.

Step 2. $\Sc^*(Z^n)^{\tG,\chi}=0$.\\
Consider the space $Y:=Z^n \times F^{\times}$ and let the group
$GL_n \times GL_n$ act on it by $(h_1,h_2) (z,\lambda):=
((h_1,h_2)z,\det h_1 \det h_2^{-1} \lambda)$. Extend this action
to action of $\widetilde{GL_n \times GL_n}$ by $\sigma(z,\lambda)
:= (\sigma(z),\lambda)$. Consider the projection $Z^n \times
F^{\times} \to F^{\times}$. By Frobenius reciprocity (theorem
\ref{padic-Frob}),
$$\Sc^*(Y)^{\widetilde{GL_n \times GL_n},\chi} \cong
\Sc^*(Z^n)^{\tG,\chi}.$$ Let $Y'$ be equal to $Y$ as an $l$-space
and let $\widetilde{GL_n \times GL_n}$ act on $Y'$ by $(h_1,h_2)
(z,\lambda):= ((h_1,h_2)z,\lambda)$ and
$\sigma(z,\lambda):=(\sigma(z),\lambda)$. Now $Y$ is isomorphic to
$Y'$ as a $\widetilde{GL_n \times GL_n}$ space by
$((A,v,\phi),\lambda) \mapsto ((A,v,\phi), \lambda \det A^{-1})$.

Since $\Sc^*(Z^n)^{\widetilde{GL_n \times GL_n},\chi} = 0$,
proposition \ref{padic-Product} implies that
$\Sc^*(Y')^{\widetilde{GL_n \times GL_n},\chi} = 0$ and hence
$\Sc^*(Y)^{\widetilde{GL_n \times GL_n},\chi} = 0$ and thus
$\Sc^*(Z^n)^{\tG_n,\chi}=0$.%
\end{proof}

\begin{corollary}
We have $$\Sc^*(W_i \times W_{n-i})^{H_i \times H_{n-i}} =
\Sc^*(W_i \times W_{n-i})^{\tH_i \times \tH_{n-i}}.$$
\end{corollary}
\begin{proof}
It follows from the key lemma and proposition \ref{padic-Product}.
\end{proof}

Now we are ready to prove proposition \ref{padic-Enough}.
\begin{proof}[Proof of proposition \ref{padic-Enough}]
Fix $i<n$. Consider the projection $pr_1:Z^i \to Q^i$. It is easy
to see that the action of $\tG$ on $Q^i$ is transitive. We are
going to use Frobenius reciprocity.

Denote
$$A_i := \left(
  \begin{array}{cc}
    Id_{i\times i} & 0 \\
    0 & 0 \\
  \end{array}
\right) \in Q^i.$$ Denote by $G_{A_i}:= \mathrm{Stab}_G(A_i)$ and
$\tG_{A_i}:= \mathrm{Stab}_{\tG}(A_i)$.

It is easy to check by explicit computation that
\begin{itemize}
\item $G_{A_i}$ and $\tG_{A_i}$ are unimodular.
\item
$H_i \times G_{n-i}$ can be canonically
embedded into $G_{A_i}$. %$\tG_{A_i}$ is canonically embedded into $\tH_i \times \tG_{n-i}$,
\item
$W$ is isomorphic to $W_i \times W_{n-i}$ as $H_i \times
G_{n-i}$-spaces.
\end{itemize}

By Frobenius reciprocity (theorem \ref{padic-Frob}),
$$\Sc^*(Z^i)^{\tG,\chi}=\Sc^*(W)^{\tG_{A_i},\chi}.$$
Hence it is enough to show that
$\Sc^*(W)^{G_{A_i}}=\Sc^*(W)^{\tG_{A_i}}.$ Let $\xi \in
\Sc^*(W)^{G_{A_i}}$. By the previous corollary, $\xi$ is $\tH_i
\times \tH_{n-i}$-invariant. Since $\xi$ is also
$G_{A_i}$-invariant, it is $\tG_{A_i}$-invariant.
\end{proof}
\subsection{Proof of the key lemma (lemma \ref{padic-key})} \label{padic-ProofKey}
$ $

Our key lemma is proved in section 10.1 of \cite{RS}. The proof
below is slightly different and more convenient to adapt to the
archimedean case.

\begin{proposition} \label{1Enough}
It is enough to prove the key lemma for $n=1$.
\end{proposition}
\begin{proof}
%\footnote{This proposition is an adaption of a statement in \cite{RS}.}
Consider the subgroup $T_n
\subset H_n$ consisting of diagonal matrices, and
$\widetilde{T}_n:=T_n\rtimes S_2 \subset \tH_n$. It is enough  to
prove $\Sc^*(W_n)^{\widetilde{T}_n,\chi}=0.$

Now, by proposition \ref{padic-Product} it is enough to prove
$\Sc^*(W_1)^{\tH_1,\chi}=0$.
\end{proof}
From now on we fix $n:=1$, $H:=H_1$, $\tH:=\tH_1$ and $W:=W_1$.
Note that $H = F^{\times}$ and $W=F^2$. The action of $H$ is given
by $\rho(\lambda)(x,y):=(\lambda x, \lambda^{-1} y)$ and extended
to the action of $\tH$ by the involution $\sigma(x,y) = (y,x)$.

Let $Y:=\{(x,y)\in  F^2| \, xy=0\} \subset W$ be the {\bf cross}
and $Y':=Y \setminus \{0\}$.

By proposition \ref{padic-OrbitCheck}, it is enough to prove the
following proposition.
\begin{proposition} \label{padic-StratCross}
$ $\\
(i) $\Sc^*(\{0\})^{\tH,\chi}=0$.\\
(ii) Any distribution $\xi \in \Sc^*(Y')^{\tH,\chi}$ is
homogeneous of type 1.\\
(iii) $\Sc^*(W\setminus Y)^{\tH,\chi}=0$.
\end{proposition}
\begin{proof}
(i) and (ii) are trivial.\\
(iii) Denote $U:=W \setminus Y$. We have to show
$\Sc^*(U)^{\tH,\chi}=0$. Consider the coordinate change $U \cong
F^{\times} \times F^{\times}$ given by $(x,y) \mapsto (xy,x/y)$.
It is an isomorphism of $\tH$-spaces where the action of $\tH$ on
$F^{\times} \times F^{\times}$ is only on the second coordinate,
and given by $\lambda (w) = \lambda^2 w$ and $\sigma(w) = w^{-1}$.
Clearly, $\Sc^*(F^\times)^{\tH,\chi}=0$ and hence by proposition
\ref{padic-Product} $\Sc^*(F^\times \times
F^\times)^{\tH,\chi}=0$.
\end{proof}

\section{Preliminaries on equivariant distributions in the archimedean case} \label{Prel}
From now till the end of the paper $F$ denotes an archimedean
local field, that is $\R$ or $\C$. Also, the word "smooth" means
infinitely differentiable.
\subsection{Notations} \label{Notations}
\subsubsection{Distributions on smooth manifolds}

$ $

Here we present basic notations on smooth manifolds and
distributions on them.

\begin{definition}
Let $X$ be a smooth manifold. Denote by $C_c^{\infty}(X)$ the
space of complex-valued test functions on $X$, that is smooth
compactly supported functions, with the standard topology, i.e.
the topology of inductive limit of \Fre spaces.

Denote $\cD (X):= C_c^{\infty}(X)^*$ equipped with  the weak
topology.

For any vector bundle $E$ over $X$ we denote by
$C_c^{\infty}(X,E)$ the complexification of space of smooth
compactly supported sections of $E$ and by $\cD (X,E)$ its dual
space. Also, for any finite dimensional real vector space $V$ we
denote $C_c^{\infty}(X,V):=C_c^{\infty}(X,X \times V)$ and
$\cD(X,V):=\cD(X,X\times V)$, where $X \times V$ is a trivial
bundle.
\end{definition}
\begin{definition}
Let $X$ be a smooth manifold and let $Z \subset X$ be a closed
subset. We denote $\cD_X(Z):= \{\xi \in \cD(X)|\Supp(\xi) \subset
Z\}$.

For locally closed subset $Y \subset X$ we denote
$\cD_X(Y):=\cD_{X\setminus (\overline{Y} \setminus Y)}(Y)$. In the
same way, for any bundle $E$ on $X$ we define $\cD_X(Y,E)$.
\end{definition}
\begin{notation}
Let $X$ be a smooth manifold and $Y$ be a smooth submanifold. We
denote by $N_Y^X:=(T_X|_Y)/T_Y $ the normal bundle to $Y$ in $X$. We
also denote by $CN_Y^X:=(N_Y^X)^*$ the conormal  bundle. For a point
$y\in Y$ we denote by $N_{Y,y}^X$ the normal space to $Y$ in $X$ at
the point $y$ and by $CN_{Y,y}^X$ the conormal space.
\end{notation}
We will also use notions of a cone in a vector space and of
homogeneity type of a distribution defined in the same way as in
non-archimedean case (definitions \ref{DefCone} and
\ref{DefHomType}).

\subsubsection{Schwartz distributions on Nash manifolds}
$ $

Our proof of Theorem A uses a trick (proposition
\ref{SimpleOrbitCheck}) involving Fourier Transform which cannot
be directly applied to distributions. For this we require a theory
of Schwartz functions and distributions as developed in \cite{AG}.
This theory is developed for Nash manifolds. Nash manifolds are
smooth semi-algebraic manifolds but in the present work only
smooth real algebraic manifolds are considered (section
\ref{AppFilt} is a minor exception). Therefore the reader can
safely replace the word {\it Nash} by {\it smooth real algebraic}.

Schwartz functions are functions that decay, together with all
their derivatives, faster than any polynomial. On $\R^n$ it is the
usual notion of Schwartz function. For precise definitions of
those notions we refer the reader to \cite{AG}. We will use the
following notations.

\begin{notation}
Let $X$ be a Nash manifold. Denote by $\Sc(X)$ the \Fre space of
Schwartz functions on $X$.

Denote by $\Sc^*(X):=\Sc(X)^*$ the space of Schwartz distributions
on $X$.

For any Nash vector bundle $E$ over $X$ we denote by $\Sc(X,E)$ the
space of Schwartz sections of $E$ and by $\Sc^*(X,E)$ its dual
space.
\end{notation}
\begin{definition}
Let $X$ be a smooth manifold, and let $Y \subset X$ be a locally
closed (semi-)algebraic subset. Let $E$ be a Nash bundle over $X$.
We define $\Sc^*_X(Y)$ and  $\Sc^*_X(Y,E)$ in the same way as
$\cD_X(Y)$ and $\cD_X(Y,E)$.
\end{definition}

\begin{remark}
All the classical bundles on a Nash manifold are Nash bundles. In
particular the normal and conormal bundle to a Nash submanifold of a
Nash manifold are Nash bundles. For proof see e.g. \cite{AG},
section 6.1.
\end{remark}

\begin{remark}
For any Nash manifold $X$, we have $C_c^{\infty}(X) \subset
\Sc(X)$ and $\Sc^*(X) \subset \cD(X)$.
\end{remark}

\begin{remark}
Schwartz distributions have the following two advantages over
general distributions:\\
(i) For a Nash manifold $X$ and an open Nash submanifold $U\subset
X$, we have the following exact sequence
$$0 \to \Sc^*_X(X \setminus U)\to \Sc^*(X) \to \Sc^*(U)\to 0.$$
(see Theorem \ref{ExSeq} in Appendix \ref{AppFilt}).\\
(ii) Fourier transform defines an isomorphism $\Fou:\Sc^*(\R^n)
\to \Sc^*(\R^n)$.
\end{remark}

\subsection{Basic tools}
$ $

We present here basic tools on  equivariant distributions that we
will use in  this paper. All the proofs are given in the
appendices.

\begin{theorem} \label{Strat}
Let a real reductive group $G$ act on a smooth affine real
algebraic variety $X$. Let $X = \bigcup_{i=0}^l X_i$ be a smooth
$G$-invariant stratification of $X$. Let $\chi$ be an algebraic
character of $G$. Suppose that for any $k \in \Z_{\geq 0}$ and any
$0 \leq i \leq l$ we have $\cD(X_i,Sym^k(CN_{X_i}^X))^{G,\chi}=0$.
Then $\cD(X)^{G,\chi}=0$.
\end{theorem}
For proof see appendix \ref{ProofStrat}.

\begin{proposition} \label{Product}
Let $H_i \subset G_i$ be Lie groups acting on smooth manifolds $X_i$
for $i=1 \ldots n$. Let $E_i \to X_i$ be (finite dimensional)
$G_i$-equivariant vector bundles. Suppose that
$\cD(X_i,E_i)^{H_i}=\cD(X_i,E_i)^{G_i}$ for all $i$. Then $\cD(\prod
X_i, \boxtimes E_i)^{\prod H_i}=\cD(\prod X_i, \boxtimes E_i)^{\prod
G_i}$, where $\boxtimes$ denotes the external product of vector
bundles.
\end{proposition}
The proof of this proposition is the same as of its
non-archimedean analog (proposition \ref{padic-Product}).

\begin{theorem}[Frobenius reciprocity] \label{Frob}
Let a unimodular Lie group $G$ act transitively on a smooth
manifold $Z$. Let $\varphi:X \to Z$ be a $G$-equivariant smooth
map. Let ${z_0}\in Z$. Suppose that its stabilizer
$\mathrm{Stab}_G({z_0})$ is unimodular. Let $X_{z_0}$ be the fiber
of ${z_0}$. Let $\chi$ be a character of $G$. Then
$\cD(X)^{G,\chi}$ is canonically isomorphic to
$\cD(X_{z_0})^{\mathrm{Stab}_G({z_0}),\chi}$. Moreover, for any
$G$-equivariant bundle $E$ on $X$ and a closed
$\mathrm{Stab}_G({z_0})$-invariant subset $Y \subset X_{z_0}$, the
space $\cD_X(GY,E)^{G,\chi}$ is canonically isomorphic to
$\cD_{X_{z_0}}(Y,E|_{X_{z_0}})^{\mathrm{Stab}_G({z_0}),\chi}$.
\end{theorem}
In section \ref{SecFrob} we formulate and prove a more general
version of this theorem.

%Bruhat change

The next theorem shows that in certain cases it is enough to show
that there are no equivariant Schwartz distributions. This will
allow us to use Fourier transform.

We will need the following theorem from \cite{AG_Gen_HC_RJR},
Theorem 4.0.2.

\begin{theorem}   \label{NoSNoDist}
Let a real reductive group $G$ act on a smooth affine real
algebraic variety $X$. Let $V$ be a finite-dimensional algebraic
representation of $G$. Suppose that $$\Sc^*(X,V)^{G}=0.$$ Then
$$\cD(X,V)^{G}=0.$$
\end{theorem}
For proof see \cite{AG_Gen_HC_RJR}, Theorem 4.0.2.

%For proof see subsection \ref{ProofInv}.

\subsection{Specific tools}
$ $

We present here tools on equivariant distributions which are more
specific to our problem. All the proofs are given in Appendix
\ref{AppFilt}.

\begin{proposition} \label{Trick}
Let a Lie group $G$ act on a smooth manifold $X$. Let $V$ be a real
finite dimensional representation of $G$. Suppose that $G$ preserves
the Haar measure on $V$. Let $U\subset V$ be an open non-empty
$G$-invariant subset. Let $\chi$ be a character of $G$. Suppose that
$\cD(X \times U)^{G,\chi}=0$. Then $\cD(X, Sym^k(V))^{G,\chi}=0$.
\end{proposition}
For proof see section \ref{ProofTrick}.
\begin{proposition}\label{SimpleOrbitCheck}
Let $G$ be a Nash group. Let $V$ be a finite dimensional
representation of $G$ over $F$. Suppose that the action of $G$
preserves some non-degenerate bilinear form $B$ on $V$. Let $V =
\bigcup \limits _{i=1}^n S_i$ be a stratification of $V$ by
$G$-invariant Nash cones.

Let $\mathfrak{X}$ be a set of characters of $F^{\times}$ such
that the set $\mathfrak{X} \cdot \mathfrak{X}$ does not contain
the character $|\cdot|^{\dim_{\R}V}$. Let $\chi$ be a character of
$G$. Suppose that for any $i$ and $k$, the space
$\Sc^*(S_i,Sym^k(CN_{S_i}^V))^{G,\chi}$ consists of homogeneous
distributions of type $\alpha$ for some $\alpha \in \mathfrak{X}$.
Then $\Sc^*(V)^{G,\chi}=0$.
\end{proposition}

For proof see section \ref{ProofOCheck}.

In order to prove homogeneity of invariant  distributions we will
use the following corollary of Frobenius reciprocity.

\begin{proposition}[Homogeneity criterion] \label{HomoCrit}
Let $G$ be a Lie group. Let $V$ be a finite dimensional
representation of $G$ over $F$. Let $C\subset V$ be a
$G$-invariant $G$-transitive smooth cone. Consider the actions of
$G \times F^{\times}$ on $V$, $C$ and $CN_C^V$, where $F^{\times}$
acts by homotheties. Let $\chi$ be a character of $G$. Let
$\alpha$ be a character of $F^{\times}$. Consider the character
$\chi ' := \chi \times \alpha^{-1}$ of $G \times F^{\times}$. Let
$x_0 \in C$ and denote $H:=Stab_G (x_0)$ and $H':=Stab_{G \times
F^{\times}} (x_0)$. Suppose that $G,H,H'$ are unimodular. Fix $k
\in Z_{\geq 0}$.

Then the space $\cD(C,Sym^k(CN_{C}^V))^{G,\chi}$ consists of
homogeneous distributions of type $\alpha$ if and only if
$$ (Sym^k(N_{C,{x_0}}^V) \otimes _{\R} \C)^{H,\chi}= (Sym^k(N_{C,{x_0}}^V) \otimes _{\R}
\C)^{H',\chi '}.$$
\end{proposition}
\section{Proof of Theorem A for archimedean $F$}\label{proofs}

\noindent We will use the same notations as in the non-archimedean
case (see notation \ref{NotObjects}).
%We will show that the following theorem implies theorem \ref{goal}.
Again, the following theorem implies Theorem A.

\setcounter{lemma}{0}

\begin{theorem} \label{main}
$\cD(X)^{\tG,\chi}=0$.
\end{theorem}
The implication is proven exactly in the same way as in the
non-archimedean case (subsection \ref{SecReduction}).

\subsection{Proof of theorem \ref{main}}
$ $

We will now stratify $X (=gl_n \times V \times V^*)$ and deal with
each strata separately.
\begin{notation}
Denote $W:=W_n:=V_n \oplus V_n^*.$ Denote by $Q^i:=Q^i_n \subset
\mathrm{gl}_n$ the set of all matrices of rank $i$. Denote
$Z^i:=Z^i_n:=Q^i_n \times W_n$.
\end{notation}
Note that $X = \bigcup Z^i$. Hence by theorem \ref{Strat}, it is
enough to prove the following proposition.

\begin{proposition} \label{Enough}
$\cD(Z^i, Sym^k(CN_{Z^i}^X ))^{\tG,\chi}=0$ for any $k$ and $i$.
\end{proposition}

We will use the following key lemma.
\begin{lemma} [Key Lemma] \label{key}
$\cD(W)^{\tH,\chi}=0.$
\end{lemma}
For proof see subsection \ref{ProofKey} below.

\begin{corollary}
Proposition \ref{Enough} holds for $i=n$.
\end{corollary}

The proof is the same as in the non-archimedean case (corollary
\ref{padic-in}).

\begin{corollary} \label{SymW}
$\cD(W_n, Sym^k(\mathrm{gl}_n^*))^{\tG,\chi}=0.$
\end{corollary}
\begin{proof}
Consider the Killing form $K:\mathrm{gl}_n^* \to \mathrm{gl}_n$. Let
$U:=K^{-1}(Q^n_n)$. In the same way as in the previous corollary one
can show that $\cD(W_n \times U)^{\tG,\chi}=0$. Hence by proposition
\ref{Trick}, $\cD(W_n, Sym^k(\mathrm{gl}_n^*))^{\tG,\chi}=0.$
\end{proof}

\begin{corollary}
We have $$\cD(W_i \times W_{n-i},Sym^k(0 \times
\mathrm{gl}_{n-i}^*))^{H_i \times G_{n-i}} = \cD(W_i \times
W_{n-i},Sym^k(0 \times \mathrm{gl}_{n-i}^*))^{\tH_i \times
\tG_{n-i}}.$$
\end{corollary}
\begin{proof}
It follows from the key lemma, the last corollary and proposition
\ref{Product}.
\end{proof}

Now we are ready to prove proposition \ref{Enough}.
\begin{proof}[Proof of proposition \ref{Enough}]
Fix $i<n$. Consider the projection $pr_1:Z^i \to Q^i$. It is easy to
see that the action of $\tG$ on $Q^i$ is transitive. Denote $$A_i :=
\left(
  \begin{array}{cc}
    Id_{i\times i} & 0 \\
    0 & 0 \\
  \end{array}
\right) \in Q^i.$$ Denote by $G_{A_i}:= \mathrm{Stab}_G(A_i)$ and
$\tG_{A_i}:= \mathrm{Stab}_{\tG}(A_i)$. Note that they are
unimodular. By Frobenius reciprocity (theorem \ref{Frob}),
$$\cD(Z^i,Sym^k(CN_{Z^i}^X
))^{\tG,\chi}=\cD(W,Sym^k(CN_{Q^i,A_i}^{\mathrm{gl}_n}))^{\tG_{A_i},\chi}.$$
Hence it is enough to show that
$$\cD(W,Sym^k(CN_{Q^i,A_i}^{\mathrm{gl}_n}))^{G_{A_i}}=\cD(W,Sym^k(CN_{Q^i,A_i}^{\mathrm{gl}_n}))^{\tG_{A_i}}.$$

It is easy to check by explicit computation that
\begin{itemize}
\item
$H_i \times G_{n-i}$ is canonically
embedded into $G_{A_i}$, %$\tG_{A_i}$ is canonically embedded into $\tH_i \times \tG_{n-i}$,
\item
$W$ is isomorphic to $W_i \times W_{n-i}$ as $H_i \times
G_{n-i}$-spaces
\item
$CN_{Q^i,A_i}^{\mathrm{gl}_n}$ is isomorphic to $0 \times
\mathrm{gl}_{n-i}^*$ as $H_i \times G_{n-i}$ representations.
\end{itemize}
Let $\xi \in
\cD(W,Sym^k(CN_{Q^i,A_i}^{\mathrm{gl}_n}))^{G_{A_i}}$. By the
previous corollary, $\xi$ is $\tH_i \times \tG_{n-i}$-invariant.
Since $\xi$ is also $G_{A_i}$-invariant, it is
$\tG_{A_i}$-invariant.
\end{proof}
\subsection{Proof of the key lemma (lemma \ref{key})} \label{ProofKey}
$ $

As in the non-archimedean case, it is enough to prove the key
lemma for $n=1$ (see proposition \ref{1Enough}).

From now on we fix $n:=1$, $H:=H_1$, $\tH:=\tH_1$ and $W:=W_1$.
Note that $H = F^{\times}$ and $W=F^2$. The action of $H$ is given
by $\rho(\lambda)(x,y):=(\lambda x, \lambda^{-1} y)$ and extended
to the action of $\tH$ by the involution $\sigma(x,y) = (y,x)$.

Let $Y:=\{(x,y)\in  F^2| \, xy=0\} \subset W$ be the {\bf cross}
and $Y':=Y \setminus \{0\}$.
\begin{lemma}
Every $(\tH,\chi)$-equivariant distribution on $W$ is supported
inside the cross $Y$.
\end{lemma}

The proof of this lemma is identical to the proof of proposition
\ref{padic-StratCross}, (iii).

To apply proposition \ref{SimpleOrbitCheck} (which uses Fourier
transform) we need to restrict our consideration to Schwartz
distributions.
%Bruhat change
%To allow that we use proposition \ref{InvAreSchwartz} that insures
%$\cD_W(Y)^{G,\chi} = \Sc^*_W(Y)^{G,\chi}$.
By theorem \ref{NoSNoDist}, in order to show that
$\cD_W(Y)^{\tH,\chi} = 0$ it is enough to show that
$\Sc^*(W)^{\tH,\chi}=0$ \footnote{Alternatively, one can show that
any $H$-invariant distribution on $W$ supported at $Y$ is a
Schwartz distribution since $Y$ has finite number of orbits.}. By
proposition \ref{SimpleOrbitCheck}, it is enough to prove the
following proposition.
\begin{proposition} $ $\\
(i) $\Sc^*(W\setminus Y)^{\tH,\chi}=0$.\\
(ii) For all $k \in \Z_{\geq 0}$, any distribution $\xi \in
\Sc^*(Y', Sym^k(CN_{Y'}^W))^{\tH,\chi}$ is $\R$-homogeneous
of type $\alpha_k$ where $\alpha_k(\lambda):=\lambda^{-2k}$.\\
(iii) $\Sc^*(\{0\},Sym^k(CN_{\{0\}}^W))^{\tH,\chi}=0$.
\end{proposition}
\begin{proof}
We have proven (i) in the proof of the previous lemma.\\
(ii) Fix $x_0:=(1,0)\in Y'$. Now we want to use the homogeneity
criterion (proposition \ref{HomoCrit}). Note that
$Stab_{\tH}(x_0)$ is trivial and $Stab_{\tH \times
\R^{\times}}(x_0) \cong \R^{\times}$. Note that $N_{Y',x_0}^W
\cong F$ and $Stab_{\tH \times \R^{\times}}(x_0)$ acts on it by
$\rho(\lambda)a = \lambda^{2} a$. So we have $$ Sym^k(N_{Y',x_0}^W
)=Sym^k(N_{Y',x_0}^W)^{\R^{\times},\alpha_k^{-1}}.$$ So by the
homogeneity criterion any distribution $\xi \in \Sc^*(Y',
Sym^k(CN_{Y'}^W))^{\tH,\chi}$ is $\R$-homogeneous of type
$\alpha_k$.\\
(iii) is a simple computation. Also, it can be deduced from (i)
using proposition \ref{Trick}.
\end{proof}

\appendix

\section{Frobenius reciprocity} \label{SecFrob} %\label{App}
%
%\subsection{Frobenius reciprocity} \label{SecFrob}
In this section we obtain a slight generalization of Frobenius
reciprocity proven in \cite{Bar} (section 3). The proof will go
along the same lines and is included for the benefit of the
reader. To simplify the formulation and proof of Frobenius
reciprocity we pass from distributions to generalized functions.
Note that the space of smooth functions embeds canonically into
the space of generalized functions but there is no canonical
embedding of smooth functions to the space of distributions.

\setcounter{lemma}{0}

\begin{notation}
Let $X$ be a smooth manifold. We denote by $D_X$ the bundle of
densities on $X$. For a point $x \in X$ we denote by $D_{X,x}$ its
fiber in the point x. If $X$ is a Nash manifold then the bundle
$D_X$ has a natural structure of a Nash bundle. For its
description see \cite{AG}, section 6.1.1.
\end{notation}

\begin{notation}
Let $X$ be a smooth manifold. We denote by $C^{-\infty} (X)$ the
space $C^{-\infty} (X):= \cD(X,D_X)$ of {\bf generalized
functions} on $X$. Let $E$ be a vector bundle on $X$. We also
denote by $C^{-\infty} (X,E)$ the space $C^{-\infty} (X,E):=
\cD(X,D_X \otimes E^*)$ of generalized sections of $E$. For a
locally closed subset $Y \subset X$ we denote $C^{-\infty}_X(Y):=
\cD_X(Y,D_X)$ and $C^{-\infty}_X (Y,E):= \cD_X(Y,D_X \otimes
E^*)$.

%For a Nash manifold $X$, a Nash bundle $E$ and a semi-algebraic
%subset $Y\subset X$ we denote $\G (X):=\Sc^* (X,D_X)$, $\G(X,E):=
%\Sc^*(X,D_X \otimes E^*)$ , $\G_X(Y):= \Sc^*_X(Y,D_X)$ and $\G_X
%(Y,E):= \Sc^*_X(Y,D_X \otimes E^*)$.
\end{notation}

We will prove the following version of Frobenius reciprocity.

\begin{theorem}[Frobenius reciprocity] \label{GenFrob}
Let a Lie group $G$ act transitively on a smooth manifold $Z$. Let
$\varphi:X \to Z$ be a $G$-equivariant smooth map. Let $z_0\in Z$.
Denote by $G_{z_0}$ the stabilizer of $z_0$ in $G$ and by
$X_{z_0}$ the fiber of ${z_0}$. Let $E$ be a $G$-equivariant
vector bundle on $X$. Then there exists a canonical isomorphism
$\mathrm{Fr}:C^{-\infty}(X_{z_0},E|_{X_{z_0}})^{G_{z_0}} \to
C^{-\infty}(X,E)^{G}$. Moreover, for any closed $G_z$-invariant
subset $Y \subset X_{z_0}$, $Fr$ maps
$C^{-\infty}_{X_{z_0}}(Y,E|_{X_{z_0}})^{G_{z_0}}$ to
$C^{-\infty}_{X}(GY,E)^{G}$.
\end{theorem}

First we will need the following version of Harish-Chandra's
submersion principle.

\begin{theorem}[Harish-Chandra's submersion principle] \label{HCsub}
Let $X,Y$ be smooth manifolds. Let $E \to X$ be a vector bundle.
Let $\varphi:Y \to X$ be a submersion. Then the map
$\varphi^*:C^{\infty}(X,E) \to C^{\infty}(Y,\varphi^*(E))$ extends
to a continuous map $\varphi^*:C^{-\infty}(X,E) \to
C^{-\infty}(Y,\varphi^*(E)).$
\end{theorem}
\begin{proof} By partition of unity it is enough to show for the case of trivial $E$.
In this case it can be easily deduced from \cite{WallachB1},
8.A.2.5.
\end{proof}

Also we will need the following fact that can be easily deduced
from \cite{WallachB1}, 8.A.2.9.
\begin{proposition} Let $E \to  Z$ be a vector
bundle and $G$ be a Lie group. Then there is a canonical
isomorphism $C^{-\infty}(Z,E) \to C^{-\infty}(Z \times G
,\pr^*(E))^G$, where $pr: Z \times G \to Z$ is the standard
projection and the action of $G$ on $Z \times$ G is the left
action on the $G$ coordinate.
\end{proposition}

The last two statements give us the following corollary.

\begin{corollary}
Let a Lie group $G$ act on a smooth manifold $X$. Let $E$ be a
$G$-equivariant bundle over $X$. Let $Z \subset X$ be a
submanifold such that the action map $G \times Z \to X$ is
submersive. Then there exists a canonical map $HC:
C^{-\infty}(X,E)^G \to C^{-\infty}(Z,E|_Z)$.
\end{corollary}

Now we can prove Frobenius reciprocity (Theorem \ref{GenFrob}).
\begin{proof} [Proof of Frobenius reciprocity.]
We construct the map
$\mathrm{Fr}:C^{-\infty}(X_{z_0},E|_{X_{z_0}})^{G_{z_0}} \to
C^{-\infty}(X,E)^{G}$ in the same way like in \cite{Ber} (1.5).
Namely, fix a set-theoretic section $\nu: Z \to G$. It gives us in
any point $z \in Z$ an identification between $X_z$ and $X_{z_0}$.
Hence we can interpret a generalized function $\xi \in
C^{-\infty}(X_{z_0},E|_{X_{z_0}})$ as a functional $\xi_z:
\cC(X_z,E^*|_{X_z} \otimes D_{X_z}) \to \C$, or as a map $\xi_z
:\cC(X_z,(E^* \otimes D_{X})|_{X_z}) \to D_{Z,z}$. Now define
$$\mathrm{Fr}(\xi)(f):=\int _{z \in Z} \xi_z(f|_{X_z}).$$ It is easy to see that $\mathrm{Fr}$ is well-defined.

It is easy to see that the map $HC: C^{-\infty}(X,E)^G \to
C^{-\infty}(X_{z_0},E|_{X_{z_0}})$ described in the last corollary
gives the inverse map.

The fact that for any closed $G_z$-invariant subset $Y \subset
X_{z_0}$, $Fr$ maps
$C^{-\infty}_{X_{z_0}}(Y,E|_{X_{z_0}})^{G_{z_0}}$ to
$C^{-\infty}_{X}(GY,E)^{G}$ follows from the fact that $Fr$
commutes with restrictions to open sets.
\end{proof}

\begin{corollary}
Theorem \ref{Frob} holds.
\end{corollary}
\begin{proof}
Without loss of generality we can assume that $\chi$ is trivial,
since we can twist $E$ by $\chi^{-1}$. We have
\begin{multline*}
\cD(X,E)^{G} \cong C^{-\infty} (X,E^* \otimes D_X)^{G} \cong
C^{-\infty}(X_{z_0},(E^* \otimes D_X)|_{X_{z_0}})^{G_{z_0}} \cong \\
(\cD(X_{z_0},E^*|_{X_{z_0}}) \otimes D_{Z,z_0})^{G_{z_0}}.
\end{multline*}
It is easy to see that in case that $G$ and $G_{z_0}$ are
unimodular, the action of $G_{z_0}$ on $D_{Z,z_0}$ is trivial.
\end{proof}

\begin{remark}
For a Nash manifold $X$ one can introduce the space of {\bf
generalized Schwartz functions} by $\G (X):=\Sc^* (X,D_X)$. Given
a Nash bundle $E$ one may consider the generalized Schwartz
sections $\G(X,E):= \Sc^*(X,D_X \otimes E^*)$. Frobenius
reciprocity in the Nash setting is obtained by restricting $Fr$
and yields
$$Fr:\G(X,E)^{G} \cong \G(X_z,E|_{X_z})^{G_z}.$$
The proof goes along the same lines, but one has to prove that the
corresponding integrals converge. We will not give the proof here
since we will not use this fact.
\end{remark}

\section{Filtrations on spaces of distributions} \label{AppFilt} %\label{App}

\subsection{Filtrations on linear spaces}
$ $

In what follows, a filtration on a vector space is always
increasing and exhaustive. We make the following definition:
\begin{definition}
Let $V$ be a vector space. Let $I$ be a well ordered set. Let $F^i$
be a filtration on  $V$ indexed by $i \in I$. We denote
$\mathrm{Gr}^i(V):= F^i/(\bigcup_{j<i} F^j)$.
\end{definition}

The following lemma is obvious.

\begin{lemma} \label{ob}
Let $V$ be a representation of an abstract group $G$. Let $I$ be a
well ordered set. Let $F^i$ be a filtration of  $V$ by $G$
invariant subspaces indexed by $i \in I$. Suppose that for any $i
\in I$ we have $\mathrm{Gr}^i(V)^G=0$. Then $V^G=0$. An analogous
statement also holds if we replace the group $G$ by a Lie algebra
$\g$.
\end{lemma}

\subsection{Filtrations on spaces of distributions} \label{ProofStrat}
$ $

\begin{theorem} \label{Filt2}
Let $X$ be a Nash manifold. Let $E$ be a Nash bundle on $X$. Let
$Z\subset X$ be a Nash submanifold. Then the space $\Sc^*_X (Z,E)$
has a natural filtration $F^k := F^k(\Sc^*_X (Z,E))$ such that
$F^k/F^{k-1} \cong \Sc^*(Z,E|_Z \otimes  Sym^k(CN^X_Z))$.
\end{theorem}

For proof see \cite{AG}, corollary 5.5.4.

We will also need the following important theorem
\begin{theorem} \label{ExSeq}
Let $X$ be a Nash manifold,  $U\subset X$ be an open Nash
submanifold and $E$ be a Nash bundle over $X$. Then we have the
following exact sequence
$$0 \to \Sc^*_X(X \setminus U,E)\to \Sc^*(X,E) \to \Sc^*(U,E|_U)\to 0.$$
\end{theorem}
\begin{proof}
The only non-trivial part is to show that the restriction map
$\Sc^*(X,E) \to \Sc^*(U,E|_U)\to 0$ is onto. It is done in
\cite{AG}, corollary 5.4.4.
\end{proof}

\noindent Now we obtain the following corollary of theorem
\ref{Filt2} using the exact sequence from theorem \ref{ExSeq}.

\begin{corollary} \label{Filt2Cor}
Let $X$ be a Nash manifold. Let $E$ be Nash bundle over $X$. Let $Y
\subset X$ be locally closed subset. Let $Y = \bigcup_{i=0}^l Y_i$
be a Nash stratification of $Y$.

Then the space $\Sc_X^*(Y,E)$ has a natural filtration
$F^{ik}(\Sc_X^*(Y,E))$ such that $$\mathrm{Gr}^{ik}(\Sc_X^*(Y,E))
\cong \Sc^*({Y_i},E|_{Y_i} \otimes Sym^k(CN^X_{Y_i}))$$ for all $i
\in \{1...l\}$ and $k \in \Z_{\geq 0}$.
\end{corollary}

\begin{corollary} \label{SchwartzStrat}
Let $X$ be a Nash manifold. Let $E$ be Nash bundle over $X$. Let
$Y \subset X$ be locally closed subset. Let $Y = \bigcup_{i=0}^l
Y_i$ be a Nash stratification of $Y$.

Suppose that for any $0\leq i \leq l$ and any $k \in \Z_{\geq 0}$,
we have $$\Sc^*({Y_i},E|_{Y_i} \otimes Sym^k(CN_{Y_i}^X))^G=0.$$
Then $\Sc^*_X(Y,E)^G=0.$
\end{corollary}
%

%We will need the following theorem from \cite{AG_Gen_HC_RJR},
%Theorem 4.0.2.
%
%\begin{theorem}   \label{NoSNoDist}
%Let a real reductive group $G$ act on a smooth affine real
%algebraic variety $X$. Let $V$ be a finite-dimensional algebraic
%representation of $G$. Suppose that $$\Sc^*(X,V)^{G}=0.$$ Then
%$$\cD(X,V)^{G}=0.$$
%\end{theorem}

By theorem \ref{NoSNoDist}, this corollary implies theorem
\ref{Strat}.

\subsection{Fourier transform and proof of proposition \ref{SimpleOrbitCheck}} \label{ProofOCheck}
\begin{notation} [Fourier transform]
Let $V$ be a finite dimensional vector space over $F$. Let $B$ be
a non-degenerate bilinear form on $V$. We denote by
$\Fou_B:\Sc^*(V) \to \Sc^*(V)$ the Fourier transform defined using
$B$ and the self-dual measure on $V$.
\end{notation}
We will use the following well known fact.

\begin{proposition}
Let $V$ be a finite dimensional vector space over $F$. Let $B$ be
a non-degenerate bilinear form on $V$. Consider the homothety
action $\rho$ of $F^{\times}$ on $\Sc^*(V)$. Then for any $\lambda
\in F^{\times}$ we have
$$\rho(\lambda) \circ \Fou_B = |\lambda |^{-\dim_{\R} V} \Fou_B \circ \rho(\lambda^{-1}).$$

\end{proposition}

\begin{notation}
Let $(\rho, \mathcal{E})$ be a complex representation of
$F^{\times}$. % of countable dimension.
We denote by $JH(\rho, \mathcal{E})$ the subset of characters of
$F^{\times}$ which are subquotients of $(\rho, \mathcal{E})$.
\end{notation}

We will use the following straightforward lemma.
\begin{lemma} \label{sf}
Let $(\rho, \mathcal{E})$ be a complex representation of
$F^{\times}$. % of countable dimension.
Let $\chi$ be a character of $F^{\times}$. Suppose that there
exists an invertible linear operator $A:\mathcal{E} \to
\mathcal{E}$ such that for any $\lambda \in F^{\times}$,
$\rho(\lambda) \circ A = \chi(\lambda) A \circ
\rho(\lambda^{-1})$. Then $JH(\mathcal{E})= \frac{\chi
}{JH(\mathcal{E})}$.
\end{lemma}

We will also use the following standard lemma.
\begin{lemma}
Let $(\rho, \mathcal{E})$ be a complex representation of
$F^{\times}$
of countable dimension.\\
(i) If  $JH(\mathcal{E}) = \emptyset$ then $\mathcal{E}=0$. \\
(ii) Let $I$ be a well ordered set and $F^i$ be a filtration on
$\mathcal{E}$ indexed by $i \in I$ by subrepresentations. Then
$JH(\mathcal{E})= \bigcup_{i \in I} JH(\mathrm{Gr}^i(\mathcal{E}))$.
\end{lemma}

Now we will prove proposition \ref{SimpleOrbitCheck}. First we
remind its formulation.

\begin{proposition}
Let $G$ be a Nash group. Let $V$ be a finite dimensional
representation of $G$ over $F$. Suppose that the action of $G$
preserves some non-degenerate bilinear form $B$ on $V$. Let $V =
\bigcup \limits _{i=1}^n S_i$ be a stratification of $V$ by
$G$-invariant Nash cones.

Let $\mathfrak{X}$ be a set of characters of $F^{\times}$ such
that the set $\mathfrak{X} \cdot \mathfrak{X}$ does not contain
the character $|\cdot|^{\dim_{\R}V}$. Let $\chi$ be a character of
$G$. Suppose that for any $i$ and $k$, the space
$\Sc^*(S_i,Sym^k(CN_{S_i}^V))^{G,\chi}$ consists of homogeneous
distributions of type $\alpha$ for some $\alpha \in \mathfrak{X}$.
Then $\Sc^*(V)^{G,\chi}=0$.
\end{proposition}
\begin{proof}
Consider $\Sc^*(V)^{G,\chi}$ as a representation of $F^{\times}$.
It has a canonical filtration given by corollary \ref{Filt2Cor}.
It is easy to see that $\mathrm{Gr}^{ik}(\Sc^*(V)^{G,\chi})$ is
canonically imbedded into $(\mathrm{Gr}^{ik}(\Sc^*(V))^{G,\chi}$.
Therefore by the previous lemma $JH(\Sc^*(V)^{G,\chi}) \subset
\mathfrak{X}^{-1}$. On the other hand $G$ preserves $B$ and hence
we have $\Fou_B:\Sc^*(V)^{G,\chi} \to \Sc^*(V)^{G,\chi}$.
Therefore by lemma \ref{sf} we have
$$JH(\Sc^*(V)^{G,\chi}) \subset |\cdot|^{-{\dim_{\R} V}} \mathfrak{X}.$$
Hence  $JH(\Sc^*(V)^{G,\chi}) = \emptyset$. Thus
$\Sc^*(V)^{G,\chi}=0$.
\end{proof}

\subsection{Proof of proposition \ref{Trick}} \label{ProofTrick}

$ $

The following proposition clearly implies  proposition
\ref{Trick}.
\begin{proposition} %\label{GenTrick}
Let $X$ be a smooth manifold. Let $V$ be a real finite dimensional
vector space. Let $U\subset V$ be an open non-empty subset. Let
$E$ be a vector bundle over $X$. Then for any $k \geq 0$ there
exists a canonical embedding $\cD(X,E \otimes Sym^k(V))
\hookrightarrow \cD(X \times U, E \boxtimes D_V)$.
%, where $\boxtimes$ denotes the external
%product of vector bundles.
\end{proposition}
\begin{proof}
It is enough to construct a continuous linear epimorphism $$\pi:
\cC(X \times U, E \boxtimes D_V) \twoheadrightarrow \cC(X,E
\otimes Sym^k(V)).$$

By partition of unity it is enough to do it for trivial $E$. Let
$w \in \cC(X \times U, D_V)$ and $x \in X$ we have to define
$\pi(w)(x) \in Sym^k(V).$ Consider the space $Sym^k(V)$ as the
space of linear functionals on the space of homogeneous
polynomials on $V$ of degree $k$. Define $$\pi(w)(x)(p):=\int_{y
\in V} p(y) w(x,y).$$ It is easy to check that $\pi(w) \in \cC(X,
Sym^k(V))$ and $\pi$ is continuous linear epimorphism.
\end{proof}

\end{document}